\documentclass{ifacconf}

\usepackage{graphicx}      
\usepackage{natbib}        

\usepackage[english]{babel}
\usepackage[applemac]{inputenc}
\usepackage{amssymb,epsfig,psfrag,amsmath,tabularx}
\usepackage{latexsym}
\usepackage{fancyhdr}
\usepackage{mhchem}
\usepackage{color}

\newcommand{\cL}{{\mathcal L}}
\newcommand{\cH}{{\mathcal H}}
\newcommand{\Tr}{{\text{Tr}}}

\newtheorem{remark}{Remark} 
 
\newtheorem{definition}{Definition} 
 
\newenvironment{proof}[1][Proof]{\textbf{#1.} }{\ \rule{0.5em}{0.5em}}

\newcommand\I{\mbox{I}}

\def\reais{{\rm I\kern-.17em R}} 
\def\real{\mathbb{R}} 
\def\complex{\mathbb{C}} 
\def\I{{\rm I\kern-.17em I}}
\def\Ir{{\rm I\kern-.17em I}}
\def\L{{\rm I\kern-.17em L}}

\begin{document}
\begin{frontmatter}

\title{Bounds for Input- and State-to-Output Properties of Uncertain Linear Systems
\thanksref{footnoteinfo}} 

\thanks[footnoteinfo]{Work supported by the Engineering and Physical Sciences Research Council projects EP/J012041/1. Giorgio Valmorbida is also a Fulford Junior Research fellow at Somerville College. James Anderson acknowledges funding from St John's College, Oxford. Dhruva Raman is supported by the EPSRC funded Systems Biology DTC at the University of Oxford.}

\author[First]{Giorgio Valmorbida} %
\author[First]{Dhruva Raman}%
\author[First]{James Anderson}

\address[First]{G. Valmorbida,  D. Raman and J. Anderson are with Department of Engineering Science, University of Oxford,  17 Parks Road, OX1 3PJ Oxford, UK, Email: \{giorgio.valmorbida,dhruva.raman, james.anderson\}@eng.ox.ac.uk. }

\begin{abstract}                %
~~We consider the effect of parametric uncertainty on properties of Linear Time Invariant systems. Traditional approaches to this problem determine the worst-case gains of the system over the uncertainty set. Whilst such approaches are computationally tractable, the upper bound obtained is not necessarily informative in terms of assessing the influence of the parameters on the system performance. We present theoretical results that lead to simple, convex algorithms producing parametric bounds on the  $\cL_2$-induced input-to-output and state-to-output gains as a function of the uncertain parameters. These bounds provide quantitative information about how the uncertainty affects the system. 

\vspace{10pt}

\textbf{To appear in the proceedings of the 8th IFAC Symposium on Robust Control Design - ROCOND'15.}
\end{abstract}


\end{frontmatter}

\section{Introduction} 

The application of polynomial optimisation techniques to robust control problems have been successful both in terms of sum-of-squares programs~(\cite{BPT13}) and the generalized problem of moments~(\cite{Las09}). Some of these advances were highlighted in the survey paper on Robust Control \cite[Section 6.1] {PT14}. The stability analysis of linear systems with linear dependence on uncertain parameters lying on the simplex was performed using parameter-dependent Lyapunov functions (LFs) as detailed in \cite{Che10} (see Section VI.C). In \cite{CGTV05b} it is proven that, for time-invariant systems, LFs which are quadratic in the state and homogeneous on parameters are necessary and sufficient to prove  stability of the linear system. However the resulting test in terms of homogeneous polynomials may require a LF of large degree. Polynomial LFs have also been studied by \cite{Bli04}  where, exploiting the fact that  the parameterised Lyapunov matrix is analytic on the uncertainty domain, it is demonstrated that  there exists a homogenous polynomial LF proving the stability for uncertainties in the hypercube. In \cite{OP07}, Pólya's theorem is applied to solve homogeneous inequalities on the simplex. Conditions for discrete-time linear systems with polynomial parameter-dependent LFs were presented in \cite{LA08} where the simplex as the uncertain domain is presented as a particular case. 

Dissipation inequalities for uncertain linear systems were studied considering interval matrix uncertainties in \cite{BBN03}, where the \emph{limits} on the uncertainty for which a dissipation inequality holds were investigated. In \cite{SP06b}, parameter-dependent LF were considered to compute bounds on $\mathcal{H}_\infty$ and $\mathcal{H}_2$ gains over polytopic uncertainty domains for linear systems with rational dependence on the uncertain parameters. The authors then use the slack variable method to solve the resulting set of matrix inequalities which are affine on the uncertain parameters. The same conditions for worst-case computation of $\mathcal{H}_\infty$ and $\mathcal{H}_2$ norms were also presented in \cite[Section VI.D]{Che10}.  In the case of autonomous systems, $\mathcal{L}_{\infty}$ norm bounds on nonlinear system trajectories, valid over compact sets of initial conditions, have been constructed using LF based approaches in \cite{Chesi2008}.


A parameterised property (or parameterised bounds on a property), such as the $\mathcal{H}_\infty$-norm of an uncertain system gives the designer important information for the selection of the components of a physical system (which define the system parameters) and allows them to quantify beforehand the degradation, measured as the norm of the mismatch to the target (nominal) performance.  Predefined bounds on performance degradation can be useful for robust and resilient control law design. In this context, obtaining a closed-loop performance curve that respects a specified degradation bound becomes a requirement for control synthesis.  This is in contrast to the customary approach of searching for guaranteed cost solutions which accounts for the worst case scenario given a set of uncertain parameters.  Also, in the context of gain-scheduling~(\cite{Pac94}), having a parameterised performance index for a set of given controllers may be of use when designing the scheduling parameter. 

A system is said to be robust to variations in its parameters if it is able to maintain its performance properties subject to a bounded variation of these parameters. For instance, the difference between the output signal of the nominal model and that of any perturbed model is expected to be \emph{small} in some sense. From the viewpoint of system identification, such a property can be problematic: a low degree of uncertainty in the input lends itself to a high degree of uncertainty in the parameter estimate, and the identification protocol becomes ill-conditioned. This property has been variously referred to as quantitative unidentifiability (\cite{Vajda1989}), practical unidentifiability (\cite{Holmberg1982}), and sloppiness (\cite{Brown2003}).   

Traditional methods of diagnosis (\cite{Vajda1989}) involve linearisation of the effects of parameter perturbation around a previously provided parameter estimate. However such effects are typically highly nonlinear (\cite{Hines2014}). Recent advances from various angles have been made in the construction of nonlinear approximations to regions of practically unidentifiable parameters around a nominal estimate (e.g. \cite{Raue2009, Calderhead2011}). Such advances typically rely on educated sampling of outputs over parameter space, with each sample involving the simulation of an ODE. Parameterised performance indices as provided in this paper, in contrast, provide algebraic characterisations of performance degradation in parameter space. This allows for simple construction of level sets of parameters inducing a given magnitude output perturbation, or optimisation of some system property subject to parametric constraints specifying a given maximal level of performance degradation.

This paper is organized as follows: Section~\ref{sec:problem} formulates the problem. In Section~\ref{sec:properties} we present results allowing to characterise input/state-to-output properties as functions of uncertain parameters. Section~\ref{sec:examples} illustrates the results with an analytical and numerical examples.  The paper is concluded in Section~\ref{sec:conclusion}.

\textit{Notation}. We take $\mathbb{R}^n$ to denote the $n$-dimensional Euclidean space. $\mathbb{R}_{\geq 0}$ specifies the set of non-negative real numbers. $I $ denotes the  identity matrix. For a given matrix $M \in \mathbb{R}^{n \times n}$, $Tr(M)$ is the trace of $M$,  $det(M)$ is the determinant, and $He(M) = M+M^{T}$. For a symmetric matrix, if $M \geq (\leq) 0$, then $M$ is positive (negative) semidefinite.  For $x_0 \in \real^n$ and $x(t) \in \mathcal{L}_2$, $\|x_0\|$, $\|x\|_{2}^2$ respectively denote Euclidean norm of $x_0$ and the squared $\mathcal{L}_2$-norm of $x(t)$. $\mathcal{C}^k$ is defined to be the set of continuous, $k$-times differentiable, scalar functions.


\section{Problem formulation}
\label{sec:problem}

Let $\Theta\subset \real^{n_\theta}$ define the set of uncertain parameters. Let $\theta^{*}\in \Theta$ be the fixed  \emph{nominal parameter values}. We refer to the linear system
\begin{equation}
\left \lbrace 
\begin{array}{rcl}
\dot{x}  &= &A(\theta^{*}) x +B(\theta^{*}) u \\
y &= &C(\theta^{*})x +D(\theta^{*}) u
\end{array}
\right.
\label{eq:x}
\end{equation}
with $x(0) = x_0$,  $A : \real^{n_\theta}\rightarrow \real^{n\times n}$, $B : \real^{n_\theta}\rightarrow \real^{n\times m}$, $C : \real^{n_\theta}\rightarrow \real^{p\times n}$ as the \emph{nominal system}. Consider the linear time-invariant parameter-dependent system
\begin{equation}
\left \lbrace 
\begin{array}{rcl}
\dot{\tilde{x}}  &= &A(\theta) \tilde{x} +B(\theta) u \\
\tilde{y} &= &C(\theta)\tilde{x}+D(\theta) u,
\end{array}
\right.
\label{eq:xtilde}
\end{equation}
$\tilde{x}(0) = x(0) = x_0$ (we consider the initial conditions of~\eqref{eq:x} and~\eqref{eq:xtilde}  to match for reasons that will become clear in next section). 
System \eqref{eq:xtilde} is called the \emph{uncertain system}.  

Define the signals $e(t) : = x(t) - \tilde{x}(t)$ and  $\Delta y(t) := y(t) - \tilde{y}(t)$ to obtain
\begin{equation}
\left \lbrace 
\begin{array}{rcl} \left[ \begin{array}{c} \dot{x} \\ \dot{e} \end{array} \right] &= &\left[ \begin{array}{cc} A(\theta^{*}) & 0  \\ \Delta A(\theta) & A(\theta) \end{array} \right]  \left[ \begin{array}{c} x \\ e \end{array} \right] +\left[ \begin{array}{c} B(\theta^{*}) \\ \Delta B(\theta) \end{array} \right] u \\
\left[ \begin{array}{c} y \\ \Delta y  \end{array} \right] &= &\left[ \begin{array}{cc} C(\theta^{*}) & 0  \\ \Delta C(\theta) & C(\theta) \end{array} \right] \left[ \begin{array}{c}  x  \\ e\end{array} \right]+\left[ \begin{array}{c} D(\theta^{*}) \\ \Delta D(\theta) \end{array} \right] u.
\end{array}
\right.
\label{eq:xe}
\end{equation}
where $\Delta A(\theta) := A(\theta^{*}) - A(\theta)$, $\Delta B(\theta) := B(\theta^{*}) - B(\theta)$ and  $\Delta C(\theta) := C(\theta^{*}) - C(\theta)$, $\Delta D(\theta) := D(\theta^{*}) - D(\theta)$, $x(0) = x_0$ and $e(0) = 0$.

We assume throughout that $A(\theta)$ is   Hurwitz $\forall \theta \in \Theta$. From the linearity of \eqref{eq:xe} and the triangular structure of the system matrix we have stability of the origin. 

In the following section we describe input and state-to-output properties of \eqref{eq:xe} considering the output $\Delta y$, which is the deviation from the nominal behavior. 

The output of the uncertain system~\eqref{eq:xtilde} is written in terms of $(x,e)$ as
$$\tilde{y} =\left[ \begin{array}{cc} C(\theta) & -C(\theta) \end{array} \right] \left[ \begin{array}{c}  x  \\ e\end{array} \right]  + D(\theta)u.$$

The interest in introducing~$\tilde{y}$ as above, instead of the mismatched output $\Delta y$ as in~\eqref{eq:xe}, is motivated by the fact that we may want a bound for $\|y(t)\|_2^2 - \|\tilde{y}(t)\|_2^2$ instead of $\|y(t) - \tilde{y}(t)\|_2^2 = \|\Delta y(t)\|_2^2$ as developed in this paper.

\subsection{Invariance of the output}

In order to make a distinction between the transient effects and the steady-state effects of the uncertainty on the mismatch signal $\Delta y(t)$ let us define the set of parameters that leave the steady-state output invariant with constant inputs. 

\begin{definition}[Steady-State Output Invariant Sets]
\label{def:ssOIS}
The set in the space of parameters
\begin{equation}
\Theta_{0ss}^{(i)} := \left\lbrace \theta \in \real^{n_\theta} \mid \lim_{t\rightarrow \infty} \Delta y(t) = 0 \, \forall u_i  \in \real \right\rbrace
\end{equation}
is called the \emph{steady-state output invariant set of parameters with respect to $i$}. Here, $u_i \in \mathbb{R}^m$ is a constant input vector taking the value $0$ in all components except the $i^{th}$. The set 
\begin{equation}
\Theta_{0ss} := \bigcap_{i = 1}^{m} \Theta_{0ss}^{(i)}
\end{equation}
is then defined as the \emph{steady-state output invariant set of parameters}.
\end{definition}
Let us now define the set of parameters leaving the output invariant with time-varying inputs
\begin{definition}[Output invariant set of parameters]
\label{def:OIS}
The set
\begin{equation}
\Theta_{0}^{(i)} := \left\lbrace \theta \in \real^{n_\theta} \mid   \Delta y(t) = 0 \, \forall t, u_i(t) \right\rbrace
\end{equation}
is called the \emph{invariant set  of parameters with respect to input $i$.}  Here, $u_i$ is as defined in Definition \ref{def:ssOIS}.
\begin{equation}
\Theta_{0} := \bigcap_{i = 1}^{m} \Theta_{0}^{(i)}
\end{equation}
is the \emph{invariant set of parameters}.
\end{definition}
 
\begin{prop} 
For system \eqref{eq:xe} the steady-state output invariant set of parameters as in Definition~\ref{def:ssOIS} with respect to input $i$ can equivalently be characterised by
\begin{equation}
\Theta_{0ss}^{(i)} := \left\lbrace \theta \in \real^{n_\theta} \mid  
f_{ss}(\theta) = 0
 \right\rbrace
\end{equation}
with
\begin{multline*}
f_{ss}(\theta) =  \Delta C(\theta^*) A^{-1}(\theta^*)B_i(\theta^*)  \\ +C(\theta)A^{-1}(\theta) \left[ \Delta B_{i}(\theta)  - \Delta A^{-1}(\theta)A^{-1}(\theta^*)B_i(\theta^*)\right].
\end{multline*}
\end{prop} 
\begin{proof}
Note that, for any constant scalar input $u_i$:
\begin{align*}
\dot{x} = 0 \Rightarrow \Delta y = 0, \Delta\dot{y} = 0, \ \ \ \forall \theta \in \Theta
\end{align*}
Expand both sides of the implication using \eqref{eq:xe}. Substituting the left hand side equation into the two right hand side equations and simplifying gives the result.
\end{proof}

By defining $\bar{x} = \left[\begin{array}{c} x \\e \end{array}\right]$, $\bar{y} = \left[\begin{array}{c}  y \\ \Delta y \end{array}\right]$ we denote system~\eqref{eq:xe} as
 \begin{equation}
\bar{G} : \left \lbrace 
\begin{array}{rcl} \dot{\bar{x}} &= &\bar{A}(\theta)\bar{x} +\bar{B}(\theta)u \\
\bar{y} &= &\bar{C}(\theta)\bar{x} +\bar{D}(\theta)u.
\end{array}
\right.
\label{eq:xeSHORT}
\end{equation}

\begin{prop} 
For system \eqref{eq:xe} the  output invariant set of parameters   as in Definition~\ref{def:OIS} with respect to input $i$ can equivalently be characterised by

\begin{equation}
\Theta_{0ss}^{(i)} := \left\lbrace \theta \in \real^{n_\theta} \mid  
f_{ss}(s, \theta) = 0 \, \forall s \in \complex
 \right\rbrace
\end{equation}
with
\begin{equation*}
f_{ss}(s,\theta) =   \bar{C}(\theta^{*})\left( sI - \bar{A}(\theta)\right)^{-1} \bar{B}_i(\theta) + \bar{D}_i(\theta).
\end{equation*}
\end{prop} 
\begin{proof}
A parameter is in the output invariant set if and only if the following holds:
\begin{align*}
\Delta{y} = 0 \Rightarrow \Delta \dot{y} = 0.
\end{align*}
This follows from the observation that $\Delta y  = 0$ at time zero.
Expanding both sides of the implication using \eqref{eq:xe} gives the result.
\end{proof}

\section{Parameterised Input/state-output properties}
\label{sec:properties}

We propose a set of parameterised input/state-output properties for linear systems with the goal of characterizing the deviation from the nominal  behavior.  We assume  that both nominal and uncertain system matrices are Hurwitz. No assumptions on the class of functions describing the parameter dependence, or the uncertainty set, are required. A convex optimisation formulation to solve some dissipation inequalities introduced in this section is presented in Section~\ref{sec:examples}, which restricts the dependence on the parameters to be of rational form.

\begin{prop}[Upper Bounds on State-to-Output Gain]
\label{prop:s2o}
For any initial condition $x_0 = x(0)$ and $e_0 = 0$, if there exists a function $V: \real^{n} \times \real^{n} \times \real^{n_\theta} \rightarrow \real_{\geq0}$, $V  \in \mathcal{C}^{1} $, $V(0,0,\cdot) = 0$ and functions $p_i:  \real^{n_\theta} \rightarrow \real_{\geq0}$, $i = 1,2$ and a set $\Theta\subset \real^{n_\theta}$ such that 
\begin{subequations}
\label{eq:s2oineq}
\begin{equation}
\label{eq:Vs2o}
 V(x,e,\theta)  \leq p_1(\theta) \|x\|^2 + p_2(\theta) \|e\|^2 \quad \forall \theta \in \Theta
\end{equation}
\begin{equation}
\label{eq:reachableset}
\dot{V}(x,e,\theta) \leq - (\Delta y)^{T}(\Delta y) \quad \forall \theta \in \Theta
\end{equation}
along the trajectories of~\eqref{eq:xe} (with $u\equiv0$) then
\begin{equation}
\label{eq:s2obound}
\dfrac{ \| \Delta y\|_2^2 }{\|x_0\|^2 } \leq p_1(\theta) \quad  \forall \theta \in \Theta.
\end{equation}
\end{subequations}
\end{prop}
\begin{proof}
Integrate~\eqref{eq:reachableset} over the interval $[0,T]$ to obtain
$$\int^{T}_{0} (\Delta y(\tau))^{T}(\Delta y(\tau)) d\tau \leq  V(x(0),e(0),\theta) - V(x(T),e(T),\theta).$$
Since $V(x(T),e(T),\theta)\geq 0$ we have
$$\int^{T}_{0} (\Delta y(\tau))^{T}(\Delta y(\tau)) d\tau \leq  V(x(0),e(0),\theta)$$
for every $T$.  With~\eqref{eq:Vs2o} we obtain
$$\int^{T}_{0} (\Delta y(\tau))^{T}(\Delta y(\tau)) d\tau \leq  p_1(\theta) \|x\|^2 + p_2(\theta) \|e\|^2.$$
Since the initial conditions of the system satisfy $(x(0),e(0)) = (x_0,0)$, and the relation holds for every $T\geq0$ we obtain~\eqref{eq:s2obound}.
\end{proof}

The above proposition helps in characterising sets of parameters bounding the  $\mathcal{L}_2$-norm of the mismatch output in terms of  level sets of the function $p_1(\theta)$ since
\begin{equation*} 
\theta \in   \left\lbrace \theta \in \Theta | p_1(\theta) \leq \lambda \right\rbrace \Rightarrow \int^{T}_{0} (\Delta y)^{T}(\Delta y)  d\tau \leq \lambda \|x_0\|^2.
\end{equation*}
In Section~\ref{sec:examples} an example illustrates level sets of a function $p_1(\theta)$ satisfying the above conditions.

\begin{remark}
\label{rem:p1thetastar}
Notice that we can always impose $p_1(\theta^{*})  =  0$ when computing  $p_1$, $p_2$ and $V$ to solve the inequalities \eqref{eq:s2oineq}. This fact results from the system structure and the choice of output to be the mismatch to the nominal system which implies that, whenever $\theta  = \theta^{*}$, we have $\Delta y(t)\equiv 0$ (yielding $\|\Delta y\|^2_2 = 0$) which implies $\frac{\|\Delta y\|^2_2}{\|x_0\|^2} = 0$.
\end{remark}
%

The proposition below provides conditions to lower bound the State-to-Output Gain of system~\eqref{eq:xe}.

\begin{prop}[Lower Bounds on State-to-Output Gain]
\label{prop:s2oLB}
For any initial condition $x_0 = x(0)$, if there exists a function $V_\ell: \real^{n} \times \real^{n} \times \real^{n_\theta} \rightarrow \real_{\geq0}$, $V_\ell \in \mathcal{C}^{1}$ and functions $p_{\ell i}:  \real^{n_\theta} \rightarrow \real_{\geq0}$, $i = 1,2$ and a set $\Theta\subset \real^{n_\theta}$ such that 
\begin{subequations}
\label{eq:s2oineqLB}
\begin{equation}
\label{eq:Vs2oLB}
p_{\ell 1}(\theta) \|x\|^2 + p_{\ell 2}(\theta) \|e\|^2   \leq   V_{\ell}(x,e,\theta)  \quad \forall \theta \in \Theta
\end{equation}
\begin{equation}
\label{eq:reachablesetLB}
-(\Delta y)^{T}(\Delta y)  \leq \dot{V}_\ell(x,e,\theta) \quad \forall \theta \in \Theta
\end{equation}
along the trajectories of~\eqref{eq:xe} then
\begin{equation}
\label{eq:s2oboundLB}
\dfrac{ \| \Delta y\|_2^2 }{\|x_0\|^2 } \geq p_{\ell 1}(\theta) \quad  \forall \theta \in \Theta.
\end{equation}
\end{subequations}
\end{prop}

The following proposition proposes a parameterisation for the  $\mathcal{L}_2$-induced gain of system~\eqref{eq:xe}.
 
\begin{prop}[Upper Bounds on the  $\mathcal{L}_2$-Induced Gain]
 \label{prop:Hinf}
For  $x(0) = 0$, $e(0) = 0$, if there exists a function $V: \real^{n} \times \real^{n} \times \real^{n_\theta} \rightarrow \real_{\geq0}$, $V(0,0,\cdot) = 0$, $V \in \mathcal{C}^{1}$ and  functions $\gamma_d:  \real^{n_\theta} \rightarrow \real_{\geq0}$ and $\gamma_n:  \real^{n_\theta} \rightarrow \real_{\geq0}$ and a set $\Theta\subset \real^{n_\theta}$ such that
\begin{subequations}
\label{eq:Hinfcond}
\begin{equation} 
V(x,e,\theta)  \geq 0 \quad  \forall \theta \in \Theta
\label{eq:VHinf}
\end{equation}
\begin{equation}
\label{eq:dVHinf}
\dot{V}(x,e,\theta) \leq - \gamma_d(\theta) (\Delta y)^{T}(\Delta y)  + \gamma_n(\theta) u^{T}u \quad \forall \theta \in \Theta
\end{equation}
then
\begin{equation}
\label{eq:hinf}
\dfrac{ \|\Delta y\|_2^2 }{ \|u \|_2^2 } \leq \gamma(\theta) \quad  \forall \theta \in \Theta,
\end{equation}
\end{subequations}
with $ \gamma(\theta)  = \dfrac{ \gamma_n(\theta)}{ \gamma_d(\theta)}$.
\end{prop}
\begin{proof}
Integrate~\eqref{eq:dVHinf} over the interval $[0,T]$ to obtain
\begin{multline*}
\int^{T}_{0}  \left( (\Delta y(\tau))^{T}(\Delta y(\tau))  - \gamma(\theta) (u(\tau))^{T}u(\tau)  \right) d\tau  \\ \leq  V(x(0),e(0),\theta) - V(x(T),e(T),\theta).
\end{multline*}
By hypothesis $V(0,0,\cdot) = 0$, which gives 
\begin{multline*}
\int^{T}_{0}  \left( \gamma_d(\theta)  (\Delta y(\tau))^{T}(\Delta y(\tau))  - \gamma_n(\theta) (u(\tau))^{T}u(\tau)  \right) d\tau  \\ \leq  - V(x(T),e(T),\theta).
\end{multline*}
From~\eqref{eq:VHinf} we have
\begin{equation*}
\int^{T}_{0}  \left( \gamma_d(\theta) (\Delta y(\tau))^{T}(\Delta y(\tau))  - \gamma_n(\theta)(u(\tau))^{T}u(\tau)  \right) d\tau  \leq  0 .
\end{equation*}

which is equivalent to~\eqref{eq:hinf} since \eqref{eq:VHinf}-\eqref{eq:dVHinf} hold for all $\theta \in \Theta$.
\end{proof}

\begin{remark}
Whenever $\Theta$ is a compact set, an upper bound on the guaranteed $\mathcal{H}_\infty$ cost of the uncertain system is  given by $\gamma^{*}+\max_{\Theta}\gamma({\theta})$, where $\gamma^{*}$ is the   $\mathcal{H}_\infty$ gain of the nominal system. As pointed out in Remark~\ref{rem:p1thetastar}, when solving~\eqref{eq:Hinfcond},  it is possible to impose $\gamma_n(\theta^*) = 0$ since for $\theta = \theta^*$ we have $\Delta y \equiv 0$ if $x(0) = e(0) = 0$.
\end{remark}

The next result applies to system~\eqref{eq:xeSHORT} with $\bar{D}(\theta) = 0$. Recall that the $\cH_2$-norm of system \eqref{eq:x} is
\begin{equation*}
\|\bar G(\theta^*)\|_{\cH_2}^2 = \text{Tr}[B^T(\theta^*)PB(\theta^*)]= [C(\theta)QC^T(\theta^*)]
\end{equation*}
where $P> 0$ solves $A^T(\theta^*)P+PA(\theta)+C^T(\theta^*)C(\theta)=0$ and $Q>0$ solves the dual equation. An upper-bound for $\|\bar G(\theta^*)\|_{\cH_2}^2$ can be found by replacing the equality in the Lyapunov equation with an inequality and searching for any positive definite matrix $\widehat{P}$ that solves the matrix inequality.

 \begin{prop}[Bounds on the $\mathcal{H}_2$-norm]
 \label{prop:H2}
If there exists a function  $V: \real^{n} \times \real^{n} \times \real^{n_\theta} \rightarrow \real_{\geq0}$, $V(0,0,\cdot) = 0$, $V  \in \mathcal{C}^{1}$,  and functions $q_1, q_2 : \Theta \rightarrow \real_{\geq0}$ satisfying,
\begin{subequations}
\label{eq:H2ineq}
\begin{equation} 
V(x,e,\theta) \leq q_1(\theta)\|x\|^2 + q_2(\theta)\|e\|^2 \quad  \forall \theta \in \Theta
\label{eq:VH2}
\end{equation}
\begin{equation}
\label{eq:lyapInequality1}
\dot{V}(x,e,\theta) +(\Delta y)^T\Delta y  \leq  0   \quad  \forall \theta \in \Theta
\end{equation}
then
 \begin{equation}
 \label{eq:boundH2final}
  \|\bar{G}(\theta)\|_{\cH_2}^2 \leq   q_1(\theta)  Tr \big( \bar{B}^T(\theta) \bar{B}(\theta) \big)  \quad \forall \theta \in \Theta.
 \end{equation}
  \end{subequations}
\end{prop}
\begin{proof}
Integrating \eqref{eq:lyapInequality1} over the interval $[0,T]$ and substituting $\Delta y(t) = \widehat{C}(\theta)\bar{x}(t)$, $\widehat{C}(\theta):= \left[\Delta C(\theta)  \, C(\theta) \right]$ gives
\begin{eqnarray*}
V(x(T),e(T),\theta)-V(x(0),e(0),\theta) &\le& - \int_0^T (\Delta y(\tau))^T\Delta y(\tau) d\tau\\
&=&-\|\Delta y\|_{2,[0,T]}^2.
\end{eqnarray*}
Letting $t\rightarrow \infty$ we arrive at 
\begin{equation}\label{eq:ybound} 
\|y\|_2^2 \le V(x(0),e(0),\theta).
\end{equation}
Let $\bar G(s)$ be a frequency domain representation of $\bar G$ (note that for clarity we do drop the dependence on $\theta$), then the $\cH_2$-norm is equivalently given by
\begin{eqnarray*}
\|\bar G(s)\|_{\cH_2}^2 &=& \frac{1}{2 \pi}\int_0^{\infty} \Tr \left[\bar G^*(j\omega)\bar G(j\omega)\right] d\omega  \\
&=& \Tr \left[ \int_0^{\infty} (\widehat Ce^{\bar A\tau}\bar  B)^T(\widehat Ce^{ \bar A\tau} \bar B) d\tau\right].
\end{eqnarray*}
Define $Y(\tau) := \widehat Ce^{\bar A\tau}\bar B$ and let $y_i$ denote the $i^\text{th}$ column of $Y$, thus $y_i = \widehat Ce^{\bar A\tau} \bar B_i$. Let $x_i(\tau)=e^{\bar A\tau}\bar B_i$, then it follows that $y_i(\tau)=\widehat Cx_i(\tau)$ and therefore
\begin{equation*}
\|\bar G\|_{\cH_2}^2= \sum_{i=1}^m \|y_i(\tau)\|^2. 
\end{equation*}
The bound \eqref{eq:ybound} gives
\begin{eqnarray}
\|y_i\|^2 \le V(x_i(0),e_0,\theta) &=& V(e^{\bar A\tau}\bar B_i,e(0),\theta) \nonumber \\
&=&V(\bar B_i,e(0),\theta)\nonumber \\
&\le&q_1(\theta)\bar B_i^T\bar B_i \label{eq:q_ineq}
\end{eqnarray}
where the final inequality comes from $e(0)=0$ and \eqref{eq:VH2}. Finally, using \eqref{eq:q_ineq} it follows that 
\begin{equation*}
\sum_{i=1}^m \|y_i\|\le q_1(\theta) \sum_{i=1}^m \bar B_i^T\bar B_i = q_1(\theta)\Tr [\bar B^T\bar B].
\end{equation*}
which upon inserting the dependence on $\theta$ completes the proof.
\end{proof}
\begin{cor}[Bounds on the $\mathcal{H}_2$-norm]
 \label{prop:H2}
If there exists a parameter dependent matrix function $P : \Theta \rightarrow \real^{2n \times 2n}$ and a function $q_1 : \Theta \rightarrow \real_{\geq0}$ satisfying, for all $\theta \in \Theta$,
\begin{subequations}
\label{eq:H2ineq}
\begin{equation}
0 \leq P(\theta)  \leq q_1(\theta) I
\end{equation}
\begin{equation}
\label{eq:lyapInequality}
\bar{A}^T(\theta)P(\theta) + P(\theta)\bar{A}(\theta) +\bar{C}(\theta)^T\bar{C}(\theta)  \leq  0 
\end{equation}
then
 \begin{equation}
 \label{eq:boundH2final}
  \|\bar{G}(\theta)\|_{\mathcal{H}_2}^2 \leq  q_1(\theta)  Tr \big( \bar{B}^T(\theta) \bar{B}(\theta) 
 \quad \forall \theta \in \Theta.
 \end{equation}
  \end{subequations}
 \end{cor}

\section{Examples}
\label{sec:examples}
\subsection{Analytic Example}
Consider system
\begin{equation}
\label{eq:sys_example}
\left\lbrace\begin{array}{l}
\dot{x} = - \dfrac{\theta_1^*}{1+ \theta_2^*}x \\
y = x
\end{array}
\right.
\end{equation}
with $\theta_i^* \in \real_{> 0}$, $i \in \{1,2\}$. The above system corresponds to \eqref{eq:x} with $A(\theta^*) = - \frac{\theta_1^*}{1+ \theta_2^*}$, $B(\theta^*) = 0$  and $C(\theta^*) = 1$. Consider the uncertainty set $$\Theta = \left\lbrace\theta \in \real^{2} | \theta_i \in (0,+\infty) \right\rbrace$$ since $\theta_1>0$ and $1+\theta_2>0$, $\forall \theta \in \Theta$, $A(\theta)$ is Hurwitz for all $\theta \in \Theta$. Now define 
\begin{equation*}
\tilde{\theta}_i : = \left( \frac{\theta_i}{\theta_i^{*}} -1 \right) \quad i \in \{1,2\}
\end{equation*}
establishing a bijection between $\Theta$ and the set $$\tilde{\Theta} = \left\lbrace \tilde{\theta} \in \real^{2} | \tilde{\theta}_i \in (-1,+\infty) \right\rbrace.$$

We can then construct system \eqref{eq:xe} for \eqref{eq:sys_example} in terms of $\tilde{\theta}$ as
\begin{equation*}
\left\lbrace \begin{array}{l}
\left[\begin{array}{c} \dot{x} \\ \dot{e} \end{array}\right] = \left[\begin{array}{cc} A(\theta^{*}) & 0 \\  A(\theta^{*})  \dfrac{( \theta_2^{*}(\tilde{\theta}_2 -\tilde{\theta}_1) - \tilde{\theta}_1  )}{1+  \theta_2^{*}(1 +\tilde{\theta}_2)}& - \dfrac{\theta_1^{*} (1 + \tilde{\theta}_1)}{ 1+  \theta_2^{*}(1 +\tilde{\theta}_2)} \end{array}\right] \left[\begin{array}{c} x \\ e \end{array}\right] \\
\left[\begin{array}{c} y \\ \Delta y \end{array}\right] = \left[\begin{array}{c} x \\ e \end{array}\right].
\end{array}\right. 
\end{equation*}

In order to obtain bounds on state-to-output gain as a function of the parameters, following Proposition~\ref{prop:s2o}, let us consider the structure
\begin{equation}
V(x,e,\tilde{\theta}) = p_1(\tilde{\theta})x^2 + p_2(\tilde{\theta})e^2
\end{equation}
with $p_i: \tilde{\Theta} \rightarrow \real_{>0}$. With the above, and using the fact that $ e = \Delta y$, inequality~\eqref{eq:reachableset} is written as $\dot{V}(x,e,\theta) + e^2 \\ = \left[\begin{array}{cc} x & e \end{array}\right]   M(\tilde{\theta})\left[\begin{array}{c} x \\ e \end{array}\right]\leq0$ with 
\begin{multline*}
M( \tilde{\theta}) = \\  He\left( \left[\begin{array}{cc} p_1(\tilde{\theta})A(\theta^*) & p_2(\tilde{\theta})A(\theta^*) \dfrac{(\theta_2^{*}(\tilde{\theta}_2 -\tilde{\theta}_1) - \tilde{\theta}_1)}{1+  \theta_2^{*}(1 +\tilde{\theta}_2)} \\  0 & \dfrac{- p_2(\tilde{\theta}) \theta_1^{*} (1 + \tilde{\theta}_1) }{ 1+  \theta_2^{*}(1 +\tilde{\theta}_2)}+ \dfrac{1}{2}  \end{array}\right] \right).
\label{eq:example_reachableset}
\end{multline*}
Let us now compute $p_1\geq0$ and $p_2\geq0$ to satisfy $M( \tilde{\theta})\leq0$ $\forall \tilde{\theta} \in \tilde{\Theta}$ therefore verifying $\dot{V}(x,e,\theta) + e^2\leq 0$ for all $\tilde{\theta} \in \tilde{\Theta}$. First consider the element $(2,2)$ in $M(\theta^{*}, \tilde{\theta})$
\begin{equation*}
M_{(2,2)}( \tilde{\theta}) = -\dfrac{2p_2(\tilde{\theta}) \theta_1^{*} (1 + \tilde{\theta}_1) }{ 1+  \theta_2^{*}(1 +\tilde{\theta}_2)} +1.
\end{equation*}
The choice of $p_2$ as
\begin{equation}
p_2(\tilde{\theta})  =  \dfrac{ (1+  \theta_2^{*}(1 +\tilde{\theta}_2) )}{\theta_1^{*} (1 + \tilde{\theta}_1)} 
\end{equation}
satisfies $p_2(\tilde{\theta}) >0 \ \forall \tilde{\theta}\in \tilde{\Theta}$ and yields $M_{(2,2)}(\tilde{\theta}) = -1$. With the above choice for $p_2$ we have
\begin{equation}
det(M(\tilde{\theta})) = 
- p_1(\tilde{\theta})A(\theta^*) - \left( A(\theta^*) \dfrac{(\theta_2^{*}(\tilde{\theta}_2 -\tilde{\theta}_1) - \tilde{\theta}_1)}{\theta_1^* (1+ \tilde{\theta}_1)}  \right)^{2}
\end{equation}
replacing the value for $p_2(\tilde{\theta})$ and picking
\begin{equation}
\label{eq:p1}
p_1(\tilde{\theta}) = \dfrac{1}{2}\dfrac{\theta_1^*}{(1+\theta_2^*)} \left(\dfrac{( \theta_2^* (\tilde{\theta}_2-\tilde{\theta}_1)-\tilde{\theta}_1)}{\theta_1^* (1+ \tilde{\theta}_1)} \right)^2
\end{equation}
we obtain $det(M(\tilde{\theta}))  = 0$. Since $det(M(\tilde{\theta}))  = 0$ and $M_{(2,2)}<0$, the matrix $M(\tilde{\theta})$ in the quadratic form $\dot{V}(x,e,\theta) + e^2$ is negative semidefinite and we have $\dot{V}(x,e,\theta) + e^2\leq 0 \, \forall \tilde{\theta} \in \tilde{\Theta}$.
Notice $p_1(\tilde{\theta})$ is non-negative $\forall \tilde{\theta}\in \tilde{\Theta}$ and $p_1(\tilde{\theta})\equiv 0$ in the subspace
$$\tilde{\Theta}_n = \left\lbrace \tilde{\theta} \mid  \left[\begin{array}{cc}1& - \frac{\theta_2^*}{(1+\theta_2^*)} \end{array}\right]\left[\begin{array}{c}\tilde{\theta}_1\\ \tilde{\theta}_2 \end{array}\right] = 0 \right\rbrace.$$
Following Proposition~\ref{prop:s2o} we have that
$$\int^{T}_{0} (\Delta y)^{T}(\Delta y) d\tau \leq  V(x(0),0,\theta) = p_1(\tilde{\theta})x^2(0),$$
from the fact that $p_1(\tilde{\theta})\equiv 0$ $\forall \tilde{\theta} \in \tilde{\Theta}_n$ we conclude that  $\int^{T}_{0} (\Delta y)^{T}(\Delta y) d\tau =0$ $\forall \tilde{\theta} \in \tilde{\Theta}_n$.

We depict $p_1(\tilde{\theta})$, the upper bound to $\|\Delta y\|^2_2$, in Figure~\ref{fig:p1} and level sets of~$p_1(\tilde{\theta})$ in Figure~\ref{fig:p2}. 

\begin{figure}[!htb]
\begin{psfrags}
     \psfrag{theta2}[l][l]{\footnotesize $\tilde{\theta}_2$}
     \psfrag{theta1}[l][l]{\footnotesize $\tilde{\theta}_1$}
     \psfrag{p1}[c][l][1][-90]{\footnotesize $p_1(\tilde{\theta})$}     
\epsfxsize=6.5cm
\centerline{\epsffile{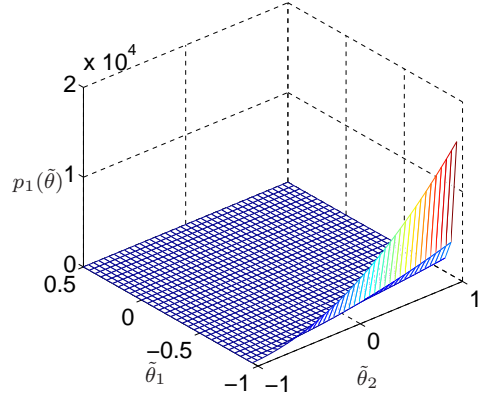}}
\end{psfrags}
\caption{Upper bound for $\frac{\|\Delta y(t)\|_2^2}{\|x(0)\|^2}$ given by $p_1(\tilde{\theta})$ as in~\eqref{eq:p1} for $\tilde{\theta} \in [-1,0.5] \times [-1,1]$. \label{fig:p1}}
\end{figure}

\begin{figure}[!htb]
\begin{psfrags}
     \psfrag{p1}[l][l][1][-90]{\footnotesize $p_1(\tilde{\theta})$}
     \psfrag{theta2}[l][l]{\footnotesize $\tilde{\theta}_2$}
     \psfrag{theta1}[c][l][1][-90]{\footnotesize $\tilde{\theta}_1$}
\epsfxsize=7cm
\centerline{\epsffile{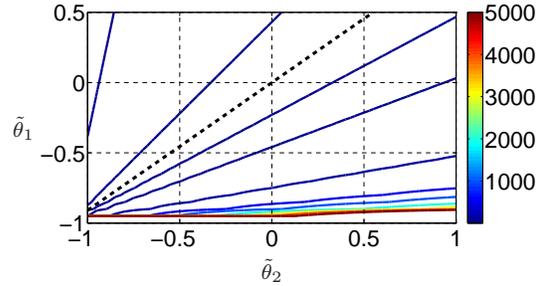}}
\end{psfrags}
\caption{Level sets of $p_1(\tilde{\theta})$. The set $\tilde{\Theta}_n$ where $p_1(\tilde{\theta})=0$ is depicted by the dashed black line. The level set corresponding to the first conical region neighboring the dashed line is $\{\theta \mid p_1(\tilde{\theta}) \leq 1\}$.
\label{fig:p2}}
\end{figure}

For this simple linear scalar system it is actually possible to compute the $\mathcal{L}_2$-norm of the output, and to obtain the state-to-output gain as:
\begin{equation}
\label{eq:p1n}
\begin{array}{rcl}
p_{1n}(\tilde{\theta}) &= & \dfrac{\|\Delta y(t)\|_2^2}{\|x(0)\|^2} \\
 &= &(\|x(0)\|^2)^{-1}\int_0^{\infty} ( e^{A(\theta^{*})\tau} x_0 - e^{A(\theta)\tau}x_0)^2 d\tau \\
&= & \int_0^{\infty}  e^{2A(\theta^{*})\tau}  - 2e^{\left(A(\theta^{*})+ A(\theta)\right)\tau} + e^{2A(\theta)\tau}  d\tau \\
&= & -\dfrac{1}{2A(\theta^{*})} + \dfrac{2}{\left(A(\theta^{*})+ A(\theta)\right)} -\dfrac{1}{2A(\theta)} \\
&= & \frac{(1+ \theta_2^*)}{2\theta_1^*} - \frac{2(1+ \theta_2^*)(1+ \theta_2^*(1 +\tilde{\theta}_2))}{\left( \theta_1^*(1+ \theta_2^*(1 +\tilde{\theta}_2)) + \theta_1^*(1 +\tilde{\theta}_1)(1+ \theta_2^*)  \right)} \\ && +\frac{1+ \theta_2^*(1 +\tilde{\theta}_2)}{2\theta_1^*(1 +\tilde{\theta}_1)}.
\end{array}
\end{equation}
The above expression therefore characterises the state-to-output gain of system~\eqref{eq:sys_example} as a rational function of the parameters $\tilde{\theta}$.  Figure~\ref{fig:Delta} depicts that $p_1$ in~\eqref{eq:p1} is an upper bound to $p_{1n}$.

\begin{figure}[!htb]
\begin{psfrags}
     \psfrag{theta2}[l][l]{\footnotesize $\tilde{\theta}_2$}
     \psfrag{theta1}[l][l]{\footnotesize $\tilde{\theta}_1$}
     \psfrag{p1}[r][r][1][-90]{\footnotesize $\begin{array}{c}p_1(\tilde{\theta}), \\ p_{1n}(\tilde{\theta}) \end{array}$}     
\epsfxsize=6.5cm
\centerline{\epsffile{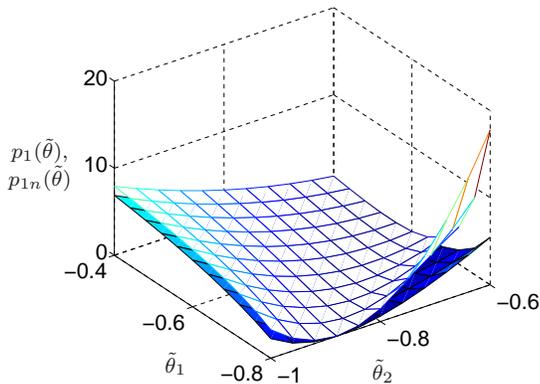}}
\end{psfrags}
\caption{Surface~$p_{1n}(\tilde{\theta})$ as in~\eqref{eq:p1n} and its upper bound~$p_{1}(\tilde{\theta})$ as in~\eqref{eq:p1}.
\label{fig:Delta}}
\end{figure}

\subsection{Numerical example}
In this example we make use of polynomial optimisation techniques to numerically construct parameter-dependent bounds on the $\mathcal{L}_2$-induced norm of an uncertain system.

Consider system
\begin{equation}
\label{eq:sys_example}
\left\lbrace\begin{array}{l}
\dot{x} = (-3 + \theta_1 +\theta_2)x + (-1 +\theta_2)u\\
y = (2 -\theta_1)x
\end{array}
\right.
\end{equation}
with $\theta^{*} = \left(0,0\right)$ and the uncertainty set defined as the rectangle $\Theta = \left[-1.5, 1.5\right] \times \left[-1, 1\right]$. In order to compute the $\mathcal{L}_2$-induced gain by checking the condition on Proposition~\ref{prop:Hinf} we consider the following constraints (see~\cite[p. 115]{BPT13} for further details)
\begin{subequations}
\begin{equation}
V(x,e,\theta) -m_{11}(x,e,\theta)p_{\theta_1}-m_{12}(x,e,\theta)p_{\theta_2} \in  \Sigma[x,e,\theta] 
\end{equation} 
\begin{multline}
-\dot{V}(x,e,\theta) - \gamma_d(\theta) (\Delta y)^{T}(\Delta y)  + \gamma_n(\theta) u^{T}u  \\ -m_{21}(x,e,\theta)p_{\theta_1}-m_{22}(x,e,\theta)p_{\theta_2} \in  \Sigma[x,e,\theta] \label{eq:dotVSOS}\end{multline} 
\begin{equation}
m_{ij}(x,e,\theta)  \in  \Sigma[x,e,\theta] \quad i,j \in \{1,2\} 
\end{equation}
\label{eq:SOS}
\end{subequations}
\noindent
with polynomials $p_{\theta_1} = -(\theta_1 +1.5)(\theta_1 -1.5)$, $p_{\theta_2} = -(\theta_1 +1)(\theta_1 -1)$, which yields $\Theta = \{\theta \mid p_{\theta_1}\geq 0 , p_{\theta_2}\geq 0 \}$ to formulate the following Sum-of-Squares program
\begin{equation}
\mbox{minimize} \, \gamma_n(0) \quad \mbox{subject to} \,  \eqref{eq:SOS}.
\label{eq:SOSP}
\end{equation}
We set polynomials $V$ and $m_{ij}$ to be quadratic on variables $x$ and $e$ and respectively of degree $3$ and $2$ on variable $\theta$. We formulate~\eqref{eq:SOSP} with SOSTOOLS~(\cite{PAVPSP13}) and solve the underlying SDP with SDPT3~(\cite{TTT99}). Figure~\ref{fig:Hinf} depicts solutions to the convex optimisation Problem~\eqref{eq:SOSP}. The guaranteed cost $\bar{\gamma} = \sqrt{4.6}$ over the set of uncertainties, is obtained by imposing $deg(\gamma_n) = deg(\gamma_d) = 0$ in the dissipation inequality~\eqref{eq:dotVSOS} and is depicted by the light gray hyperplane. The darker gray hyperplane, which sits underneath the guaranteed cost, corresponds to an affine $\gamma_n$  ($deg(\gamma_n)  = 1$) and a constant $\gamma_d$. The gain curve that corresponds to the case $deg(\gamma_n) = 2, deg(\gamma_d) = 1$ is a better approximation for the actual gain, depicted in the bottom graph, its level sets are depicted on the plane $(\theta_1,\theta_2)$. Notice that there is no guarantee that the computed surfaces have empty intersection since they were computed independently.


\begin{figure}[!htb]
\begin{psfrags}
     \psfrag{th2}[l][l]{\footnotesize $\theta_2$}
     \psfrag{th1}[l][l]{\footnotesize $\theta_1$}
     \psfrag{gam}[r][r][1][-90]{\footnotesize $\gamma(\theta)$}   
     \psfrag{GC}[l][l]{\footnotesize GC}   
     \psfrag{AC}[l][l]{\footnotesize {\color{white} AC}}        
     \psfrag{RC}[l][l]{\footnotesize RC}        
     \psfrag{TC}[l][l]{\footnotesize TC}        
\epsfxsize=9.0cm
\centerline{\epsffile{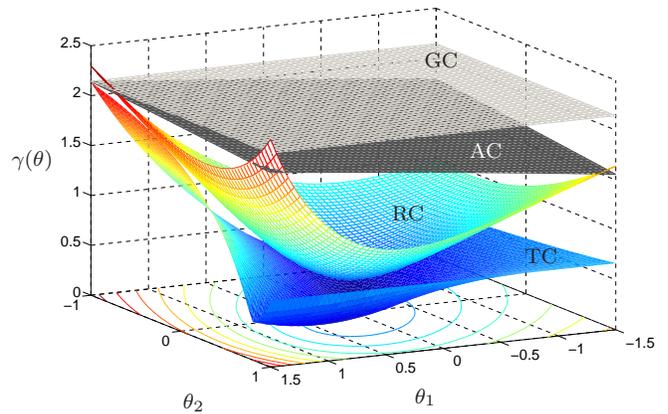}}
\end{psfrags}
\caption{ Graphs of the computed $\gamma(\theta)  = \frac{\gamma_n(\theta) }{\gamma_d(\theta)}$ for polynomials $\gamma_n(\theta)$ and $\gamma_d(\theta)$ of different degrees. The depicted surfaces correspond to  GC: Parameter-Independent Upper Bound, AC:  Affine  Parameter-Dependent Upper Bound, RC:  Rational  Parameter-Dependent Upper Bound, TC:  Actual Cost.
\label{fig:Hinf}}
\end{figure}

\section{Discussion and Conclusions}
\label{sec:conclusion}
We presented conditions to obtain upper-bounds to input/state-to-output gains for uncertain linear systems. These upper-bounds are functions of the uncertain parameters. For systems with polynomial and rational dependence on the parameters it is possible to compute these upper-bounds by solving a set of polynomial inequalities. These inequalities can then be solved with semidefinite programming.  We also provide a single-dimensional example where the exact characterization of the state-to-output gain is possible and we compare this gain with an analytic solution to the proposed conditions. 



An extension of the formulation presented is the state-feedback control design with a  \emph{given} parameterised upper bound of the input-output gains. Such a parameterised upper bound represents an allowable degradation of the nominal trajectories.  Another challenge is to assess the effects of uncertainty in the control gains while computing the control laws. Criteria for \emph{non-fragile control laws}~\cite{KB97} can be defined in terms of the upper-bounds on the system gains parameterised in terms of uncertainties in the control law gains. Quadratic separation and ellipsoidal sets for controller uncertainties were used in \cite{PA05} in this context.

The extension to nonlinear polynomial systems (systems defined by polynomial vector fields) are currently under study. For this class further assumptions are required since the existence of a single equilibrium point may not be guaranteed with the parameter variation. One important aspect is that the cascaded structure of the system when described in terms of the original state $x$ and the error $e$ as in~\eqref{eq:xe}, is preserved. However, the fact that the stability of the origin of the elements of the cascade implies the stability of the origin of the overall system is no longer true if the system under analysis is non-linear.

\appendix
\section{Sum-of-squares polynomials}
Let us take $\real[x]$ to be the ring of polynomials with real coefficients in the vector $x\in \real^n$ of indeterminates. Define the subring $\Sigma[x]$ be the set of sums of squares of polynomials in $\real[x]$:
\begin{align*}
&\Sigma[x] = \Big\{ p \in \real[x] \mid \exists m \in \mathbb{N}, \{g_i\}_{i=1}^m \in \real[x] : p = \sum_{i=1}^m g_i^2 \Big\}.
\end{align*}   

%
%
%
%
                                                                         
\end{document}